\UseRawInputEncoding
\documentclass[12pt]{article}
\usepackage{amssymb,amsmath}
\usepackage{latexsym,bm}
\usepackage{dsfont}

\usepackage{cite}
\usepackage{graphicx,amssymb,mathrsfs,amsmath,latexsym,amsfonts,amsthm}
\usepackage{amsmath,amssymb,amsthm}
\usepackage{latexsym}
\usepackage{amssymb}
\usepackage{stmaryrd}
 \usepackage{float}
\usepackage{psfrag}
\usepackage{xcolor}
\usepackage{enumerate}
\usepackage{manfnt,mathrsfs}
\usepackage{epstopdf}

\makeatletter 
\@addtoreset{equation}{section}
\makeatother
\newtheorem{theorem}{Theorem}

\newtheorem{lemma}[theorem]{Lemma}

\newtheorem{problem}[theorem]{Problem}
\def\adots{\mathinner{\mkern2mu\raise0pt\hbox{.}  
\mkern2mu\raise4pt\hbox{.}\mkern1mu
\raise7pt\vbox{\kern7pt\hbox{.}}\mkern1mu}}

\newcommand{\V}{\mathcal{V}}

\def\tu{{\rm Tur{\'a}n\,}}
\date{ }
\begin{document}
\title{Extremal digraphs containing at most $t$ paths of length 2 with the same endpoints}
\author{ Zejun Huang \thanks{ School of Mathematical Sciences, Shenzhen University, Shenzhen 518060, P.R. China. (mathzejun@gmail.com)}, Zhenhua Lyu\thanks{ School of Science, Shenyang Aerospace University, Shenyang, 110136, China. (lyuzhh@outlook.com)}  } \maketitle

\begin{abstract}
Given a positive integer  $t$, let $P_{t,2}$ be the digraph consisting of $t$ directed paths of length 2 with the same initial and  terminal vertices.
In this paper, we study  the maximum  size of $P_{t+1,2}$-free digraphs of order $n$, which is denoted by  $ex(n, P_{t+1,2})$.  For sufficiently large $n$, we prove that $ex(n, P_{t+1,2})=g(n,t)$ when $\lfloor(n-t)/{2} \rfloor$ is odd and $ex(n, P_{t+1})\in \{g(n,t)-1, g(n,t)\}$  when $\lfloor(n-t)/{2} \rfloor$ is even, where $g(n,t)=\left\lceil(n+t)/{2}\right\rceil \left\lfloor(n-t)/{2}\right\rfloor+tn+1$.
\end{abstract}


\section{ Introduction and main results}
For a digraph  $D$ with  vertex set $V$ and arc set $A$,  we call $|V|$  its
{\it order} and $|A|$ its {\it size}.
We denote by  $uv$ the arc from $u$ to $v$.  A sequence of
consecutive arcs $v_1v_2, v_2v_3,\ldots, v_{t-1}v_t$ is called a {\it directed walk}   from $v_1$ to $v_t$, denoted $v_1v_2v_3 \cdots v_{t-1}v_t$. A {\it directed path}  is a directed walk whose vertices are distinct. A {\it directed cycle} is a closed directed walk $v_1v_2\cdots v_tv_1$ with $v_1,v_2\ldots,v_t$ being distinct. A directed walk (or cycle, path)  of length $k$ is called a {\it $k$-walk (or $k$-cycle, $k$-path)}.  If two arcs between a pair of vertices  have the same direction, we say they are {\it parallel}. Digraphs in this paper are strict, i.e., they do not allow loops or parallel arcs.

	
The Tur\'an type problem plays      an important role in graph theory. It concerns the largest possible number of edges in graphs without given subgraphs and the extremal graphs achieving that maximum number of edges. The study of these  problems can be dated back to 1907, when Mantel \cite{WM} characterized   triangle-free graphs  attaining the maximum size. In \cite{PT,PT2}, Tur\'an initiated this kind of problems by generalizing
Mantel's theorem to  graphs containing no complete subgraph of a fixed order. Most results on classical \tu type problems concern undirected graphs and only a few \tu type problems on digraphs have been investigated. 

The systematic investigation of extremal problems on digraphs was initiated by Brown and Harary \cite{BH}. They determined the precise maximum size and extremal digraphs for digraphs avoiding a complete digraph of a given order. They also studied digraphs avoiding a tournament, a direct sum of two tournaments, or a digraph on at most 4 vertices where every two vertices are joined by at least one arc. H$\ddot{\rm a}$ggkvist and Thomassen \cite{HT}, Zhou and Li \cite{zl} characterized the extremal digraphs avoiding a  directed cycle. By using dense matrices, Brown, Erd\H os and Simonovits \cite{BES,BES2,BES3} presented asymptotic results on  extremal digraphs avoiding a family of digraphs. Tuite, Erskine and Salia \cite{tes} investigated the digraphs with a unique walk of length at most $k$ between any pair of vertices (not necessarily distinct).  Extremal digraphs
avoiding certain distinct walks or paths with the same initial vertex and terminal vertex were studied in \cite{HDT1,HDT2, HL1,HL2,HL3,HL4,HL,HZ,lyu2,lyu4,lyu3,MRT,HW}. Especially, following the direction of Brown and Harary, the authors of \cite{HL3} studied  a Tur\'an problem on an orientation of the 4-cycle. They  characterized the extremal digraphs avoiding two directed 2-paths with the same initial vertex and the same terminal vertex.  Lyu \cite{lyu} determined the  extremal digraphs avoiding an orientation of the diamond graph.

 In this paper, we study the maximum size of digraphs avoiding  a given number of 2-paths with the same initial vertex and terminal vertex.
	Denote by $P_{k,2}$  the digraph that consists of $k$   paths of length $2$ with the same initial vertex and terminal vertex; see Figure 1.
	\begin{figure}[H]
		\centering
		\includegraphics[width=2.in]{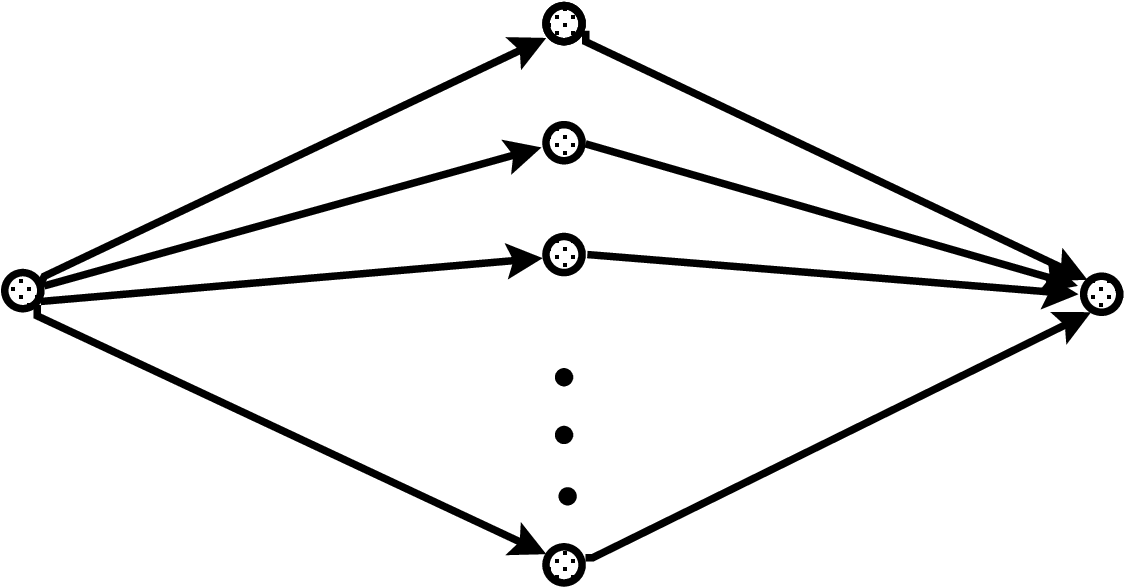}\hspace{0.8cm}\\
		Figure 1:~$P_{k,2}$
	\end{figure}
Let $ex(n,P_{k,2})$ be  the \tu number of $P_{k,2}$, i.e., the maximum size  of $P_{k,2}$-free digraphs of order $n$, and let  $EX(n,P_{k,2})$ be the set of $P_{k,2}$-free  digraphs attaining the size $ex(n,P_{k,2})$. Our problem can be formulated as follows.
\begin{problem}
Given positive integers $n$ and $k$. Determine $ex(n,P_{k,2})$.
\end{problem}

Notice that when $k=2$, the problem was solved in \cite{HL3}. We consider the cases $k\ge 3$ in this paper. Denote by
\begin{equation*}\label{eq0}
		g(n,t)=
		\left\lceil\frac{n+t}{2}\right\rceil \left\lfloor\frac{n-t}{2}\right\rfloor+tn+1.
		\end{equation*}
Our main result is the following.

\begin{theorem}\label{th1}
 Let $n,t$ be  positive integers such that  $t\ge 2$ and $n\ge \max\{t^3+4t^2+3t+4,17t^2/2+30t+27\}$.
\begin{itemize}
  \item[(i)] If $ \lfloor({n-t})/{2} \rfloor$ is  odd, then $ex(n,P_{t+1,2})=g(n,t)$;
  \item[(ii)] if $ \lfloor({n-t})/{2} \rfloor$ is even, then $ex(n,P_{t+1,2})\in\{g(n,t)-1,g(n,t)\}$.
\end{itemize}
\end{theorem}	
	

\section{Proofs}

  Let  $D=(V,A)$ be a digraph with  vertex set $V$ and arc set $A$. For $u,v\in V$, if there is an arc from $u$ to $v$, then we say $v$ is a {\it successor} of $u$, and $u$ is a  {\it predecessor} of $v$.  Given $S,T\subseteq V$,   the following notations will be used repeatedly.

\begin{itemize}
\item[]\noindent\hspace{-0.8cm}\vspace{-0.2cm}$D[S]$:  the subgraph of $D$ induced by $S$;
   \item[]   \noindent\hspace{-0.8cm}\vspace{-0.2cm}$S\Rightarrow T$:  $uv$ is an arc for all $u\in S$ and $v\in T$;
  \item[]\noindent\hspace{-0.8cm}\vspace{-0.2cm}$E(S,T)$: the set of arcs from $S$ to $T$;
  \item[] \noindent\hspace{-0.8cm}\vspace{-0.2cm}$E(S)$: the set of arcs from $S$ to $S$;
  \item[]\noindent\hspace{-0.8cm}\vspace{-0.2cm}$e(S,T)$: the cardinality of $E(S,T)$;
   \item[]\noindent\hspace{-0.8cm}\vspace{-0.2cm}$e(D)$: the size of the digraph $D$;
 \item[]\noindent\hspace{-0.8cm}\vspace{-0.2cm}$N^+(u)$: the  successor set of a vertex $u$, which is  $\{x\in V(D)|ux\in A\}$;
   \item[]\noindent\hspace{-0.8cm}\vspace{-0.2cm}$N^-(u)$:   the  predecessor set of a vertex $u$, which is  $\{x\in V(D)|xu\in A\}$;
    \item[]\noindent\hspace{-0.8cm}\vspace{-0.2cm}$N^+_{S}(u)$: the local  successor set of $u$ in $S$, which   is $\{x\in S|ux\in A\}$;
     \item[]\noindent\hspace{-0.8cm}\vspace{-0.2cm}$N^-_{S}(u)$: the local predecessor  set  of $u$ in $S$, which is $\{x\in S|xu\in A\}$;
      \item[]\noindent\hspace{-0.8cm}\vspace{-0.2cm}$N^+(T)$: the union of the successor sets of all vertices from $T$;
       \item[]\noindent\hspace{-0.8cm}\vspace{-0.2cm}$N^+_{S}(T)$: the union of the local successor sets of all vertices from $T$ in $S$;
\item[]\noindent\hspace{-0.8cm}\vspace{-0.2cm}$d^+(u)$: the  {\it outdegree}  of $u$, which is equal to $|N^+(u)|$;
   \item[]\noindent\hspace{-0.8cm}\vspace{-0.2cm}$d^-(u)$: the  {\it indegree}  of $u$,  which is equal to $|N^-(u)|$;
    \item[]\noindent\hspace{-0.8cm}\vspace{-0.2cm}$\Delta^+(D)$:  the maximum out-degree of $D$;
     \item[]\noindent\hspace{-0.8cm}\vspace{-0.2cm}$\Delta^-(D)$:  the maximum in-degree of $D$;
        \item[]\noindent\hspace{-0.8cm}\vspace{-0.2cm}$\Delta(D)$:  the maximum degree of $D$, i.e., $\Delta(D)=\max\{\Delta^+(D),\Delta^-(D)\}$;
   \item[]\noindent\hspace{-0.8cm}\vspace{-0.2cm}$\tau(u)$: the number of vertices in $D$ which are both successors and  predecessors of $u$, i.e., $$\tau(u)=|N^+(u)\cap N^-(u)|=e(N^+(u),\{u\});$$
    \item[]\noindent\hspace{-0.8cm}\vspace{-0.2cm}$\alpha(u)$: the maximum number  of local successors in $V\setminus N^+(u)$ of all vertices, i.e., $$\alpha(u)=\max\limits_{w\in V}e(\{w\},V\setminus N^+(u)).$$
     \end{itemize}
Given a vertex $u\in V$, we will also use $V_1(u)$ and $V_2(u)$ to denote $N^+(u)$ and $V\setminus V_1(u)$, respectively. Then $$\alpha(u)=\max\limits_{w\in V}e(\{w\},V_2(u)).$$
$V_1(u), V_2(u)$, $\alpha(u)$, $\tau(u)$ and $\Delta(D)$ will be abbreviated as $V_1, V_2$, $\alpha$, $\tau$ and $\Delta$  if no confusion arises.
	
For convenience, we denote by
\begin{equation*}\label{eq0}
		\phi(n,t)=
		\begin{cases}
		g(n,t)-1,&{\rm if~}\left\lfloor\frac{n-t}{2}\right\rfloor~{\rm is ~even;} \\
		g(n,t),&{\rm if~}\left\lfloor\frac{n-t}{2}\right\rfloor~{\rm is ~odd.}
		\end{cases}
		\end{equation*}
Then
\begin{equation}\label{equ1}
\phi(n,t)\ge \frac{n^2-t^2}{4}+tn-t.
\end{equation}
Now we present a lower bound on $ex(n,P_{t+1,2})$.
\begin{lemma}\label{lemma03}
Let $t\ge 2$ and  {$n\ge t+6$} be integers. Then
\begin{equation}\label{equ666}
ex(n,P_{t+1,2})\ge \phi(n,t).
\end{equation}
\end{lemma}	
\begin{proof} It is sufficient to construct a $P_{t+1,2}$-free digraph  on $n$ vertices  with size $\phi(n,t)$.  	  Let $D_1$ be a set of  digraphs on $n$ vertices whose vertex set can be divided into four parts, say $U_1,U_2,U_3,U_4$, with
$$|U_1|=\left\lfloor \frac{n-t}{2}\right\rfloor,~ |U_2|=\left\lfloor \frac{n-t-1}{2}\right\rfloor, ~|U_3|=t, ~|U_4|=1.$$
Moreover, the arc set of digraphs in $D_1$ consists of the following arcs:
\begin{itemize}
\item[(i)]
$U_1\Rightarrow U_2\cup U_3\cup U_4$, $U_3\Rightarrow U_2$;
 \item[(ii)]$D[U_3\cup U_4]$ is a complete digraph on $t+1$ vertices;
  \item[(iii)]each vertex in $U_1$ has exactly $t$ predecessors from $U_2\cup U_3$ such that $d^+_{U_1}(u)\le t$ for all $u\in U_2$ and $d^+_{U_1}(u)\le 1$ for all $u\in U_3$.
\end{itemize}

If $\lfloor (n-t)/2\rfloor$ is odd and $\lfloor (n-t-1)/2\rfloor$ is even,  let $D_2$ be a set of digraphs whose vertex set can be divided into 6 parts, say $U_1,U_2,U_3,U_4,U_5,U_6$, with
$$|U_1|=\left\lfloor \frac{n-t}{2}\right\rfloor-1,~ |U_2|=\left\lfloor \frac{n-t-1}{2}\right\rfloor,~ |U_3|=t-2,~ |U_4|=|U_5|=1,~ |U_6|=2.$$
Moreover, the arc set of digraphs in $D_2$ consists of the following arcs.
\begin{itemize}
\item[(i)] $U_1\Rightarrow U_2\cup U_3\cup U_6$, $U_4\Rightarrow U_2\cup U_3\cup U_5\cup U_6$, $U_3\cup U_5\Rightarrow U_2\cup U_4$, $ U_6\Rightarrow U_4$;
\item[(ii)] $D[U_3\cup U_5\cup U_6]$ is a complete digraph on $t+1$ vertices,
 \item[(iii)] $D[U_1]$ consists of $[\lfloor (n-t)/2\rfloor-1]/2$ disjoint 2-cycles and  $D[U_2]$ consists of $\lfloor (n-t-1)/2\rfloor/2$ disjoint 2-cycles;
  \item[(iv)]     each vertex in $U_1$ has exactly $t$ predecessors from $U_2$ such that each vertex in $U_2$ has at most $t$ successors from $U_1$.
  \end{itemize}

If $\lfloor (n-t)/2\rfloor$ and $\lfloor (n-t-1)/2\rfloor$  are both odd, i.e., $\lfloor (n-t)/2\rfloor=\lfloor (n-t-1)/2\rfloor$,  let $D_3$ be a set of digraphs whose vertex set can be divided into 6 parts, say $U_1,U_2,U_3,U_4,U_5,U_6$, with
$$|U_1|=\left\lfloor \frac{n-t}{2}\right\rfloor-1,~|U_2|=\left\lfloor \frac{n-t}{2}\right\rfloor, ~|U_3|=t-1, ~|U_4|=|U_5|=|U_6|=1.$$
Moreover, the arc set of digraphs in $D_3$ consists of the following arcs.
\begin{itemize}
\item[(i)]$U_1\Rightarrow U_2\cup U_3\cup U_5$, $U_4\Rightarrow U_2\cup U_3\cup U_5\cup U_6$, $U_3\cup U_6\Rightarrow U_2\cup U_4$ ;
\item[(ii)]$D[U_3\cup U_5\cup U_6]$ is a complete digraph on $t+1$ vertices;
 \item[(iii)]$D[U_1]$ consists of $[\lfloor (n-t)/2\rfloor-1]/2$ disjoint 2-cycles  and  $D[U_2]$ is empty;
 \item[(iv)]the vertex in $U_4$ has exactly one predecessor from $U_2$ and each vertex in $U_1$ has exactly $t$ predecessors from $U_2$ such that each vertex in $U_2$ has at most $t$ successors from $U_1\cup U_4$.
\end{itemize}

One may check case by case that all the digraphs in $D_1\cup D_2\cup D_3$ are $P_{t+1,2}$-free.	
By direct computation, when $\lfloor (n-t)/2\rfloor$ is odd, each digraph  in $ D_2\cup D_3$ contains $\phi(n,t)$ arcs; when  $\lfloor (n-t)/2\rfloor$ is even, each digraph  in $ D_1$ contains $\phi(n,t)$ arcs. Hence, we obtain the lower bound   $ex(n,P_{t+1,2})\ge \phi(n,t)$.
\end{proof}

Next we  present some lemmas   needed in the proof of Theorem \ref{th1}.

	\begin{lemma}\label{le7}
		Let $D=(V,A)$ be a $P_{t+1,2}$-free digraph and let $u\notin S\subset V$. Then
\begin{equation}\label{equ2}
e\left(W,S\right)\le t|S| \quad {  for ~all}\quad W\subseteq N^+(u).
\end{equation}
	\end{lemma}
	
	\begin{proof} Suppose $e\left(W,S\right)\ge t|S|+1$.
		Then there are $t+1$ vertices in $W$ sharing a common successor in $S$, which implies that  $D$ is not $P_{t+1,2}$-free, a contradiction. Therefore, we have \eqref{equ2}.
\end{proof}

\begin{lemma}\label{le9} Let $D=(V,A)$ be  a $P_{t+1,2}$-free  digraph and let $v\in V$.  Then  each  $u\in V_2(v)\setminus \{v\}$ has at most $d^+(v)-\tau(v)+t$  successors from  $V_1(v)$.
	\end{lemma}
\begin{proof} Suppose that there exists a vertex $u\in V_2(v)\setminus \{v\}$ with $d^+(v)-\tau(v)+t+1$  successors from $V_1(v)$. By the definition of $\tau(v)$, $u$ has at least $t+1$  distinct successors $v_1,v_2,\ldots, v_{t+1}$   from $N^+_{V_1(v)}(u)\cap N^-_{V_1(v)}(v)$. So we have the following $t+1$ paths of length 2:
		$$u\rightarrow v_i\rightarrow v \quad {\rm for}\quad i=1,2,\ldots, t+1.$$
Hence, $D$ contains a copy of $P_{t+1,2}$, a contradiction.
\end{proof}

The following lemma  is crucial.
\begin{lemma}\label{le12}
Let $t\ge 2$ and $n\ge  \max\{t^3+4t^2+3t+4,17t^2/2+30t+27\}$ be integers. Suppose $D\in EX(n,P_{t+1,2})$ with  $\Delta^+\ge \Delta^-$. Then
\begin{itemize}
  \item[(1)] for every vertex $v$ with  out-degree $\Delta$,  we have $\alpha(v)\le t$;
  \item[(2)] $\Delta=\left\lceil\frac{n+t}{2}\right\rceil$;
  \item[(3)] $\min\limits_{d^+(u)=\Delta}\tau(u)\le t+1$.
\end{itemize}

	\end{lemma}
	
	\begin{proof} Let $v$ be a vertex with out-degree $\Delta$ such that $\tau(v)=\min\limits_{d^+(u)=\Delta}\tau(u)$. Denote by
$$\tau=\tau(v), \alpha_1=\max\limits_{u\in V_1}e(u,V_2), \alpha_2=\max\limits_{u\in V_2}e(u,V_2), {\rm ~and ~}\beta=\max\limits_{u\in V_2\setminus \{v\}}d^+(u). $$
Then we have  $\alpha=\max\{\alpha_1,\alpha_2\}$.

The values of these parameters are entangled with each other so that we can not obtain each of their bounds by one step.  We follow a strategy as follows. At first we give a rough bound on $\Delta$. Using it, we get a bound on $\tau$, which can improve the bound on $\Delta$ in return. Then we  derive upper bounds on $\alpha_1$ and $\alpha_2$, which may be used as new conditions to provide sharper bounds on $\Delta$ and $\tau$.

It is obvious that
		\begin{equation}\label{equ:3.1}
		e(D)= e(V_2,V)+e(V_1,V)=\sum\limits_{u\in V_2}d^+(u)+\sum\limits_{u\in V}e(V_1,\{u\}).
		\end{equation}
		We assert that
		\begin{equation}\label{equ:3.3}
		\tau\le \Delta-\beta+\alpha_2+t.
		\end{equation}
		In fact, if $\beta\le \alpha_2+t$, (\ref{equ:3.3}) holds trivially. If $\beta> \alpha_2+t$, choose a vertex  $u\in V_2\setminus\{v\}$ with $d^+(u)=\beta$. Then $u$ has at least $\beta-\alpha_2$ successors in $V_1$. Applying Lemma \ref{le9} we get $\beta-\alpha_2\le \Delta-\tau+t,$ which leads to (\ref{equ:3.3}).	
		
Firstly we present  lower and upper bounds on $\Delta$. By the definition of $\tau$, we have
\begin{equation}\label{eq12}
\tau\le \Delta.
\end{equation}
Since  $D$ is $P_{t+1,2}$-free, every vertex in $V\setminus\{v\}$ has at most $t$ predecessors from $V_1(v)$. Hence,
\begin{equation}\label{eq2.07}
\sum\limits_{u\in V}e(V_1,\{u\})\le t(n-1)+\tau.
\end{equation}
Combining this with  (\ref{equ1}), \eqref{equ666}, (\ref{equ:3.1}) and \eqref{eq12}, we obtain
\begin{eqnarray}\label{eq3}
0&\le&\Delta(n-\Delta)+t(n-1)+\tau-\frac{n^2-t^2}{4}-tn+t\\
&\le&\Delta(n-\Delta)+t(n-1)+\Delta-\frac{n^2-t^2}{4}-t(n-1)\nonumber\\
&=&-(\Delta-\frac{n+1}{2})^2+\frac{(n+1)^2}{4}-\frac{n^2-t^2}{4}\nonumber.
\end{eqnarray}
It follows that
\begin{equation}
\frac{2n+1+t^2}{4}\ge(\Delta-\frac{n+1}{2})^2,
\end{equation}which implies
\begin{equation}\label{eq4}
\frac{n+1-\sqrt{2n+1+t^2}}{2}\le \Delta\le \frac{n+1+\sqrt{2n+1+t^2}}{2}.
\end{equation}
Recalling the bound on $n$, we have
\begin{equation}\label{eq8}
\frac{2}{5}n+\frac{1}{2}\le \Delta\le \frac{3}{5}n+\frac{1}{2}.
\end{equation}

Let $u_0\in V_2$ be a vertex  with $\alpha_2$ successors in $V_2$, i.e., $|N^+_{V_2}(u_0)|=\alpha_2$.  By (\ref{equ2}) and \eqref{eq2.07}, we have
$$ e\left(N^+_{V_2}(u_0),V\right)=e(N^+_{V_2}(u_0),V\setminus \{u_0\})+e(N^+_{V_2}(u_0), \{u_0\})\le t(n-1)+\alpha_2,$$
and
$$e(V_1,V)\le tn-t+\tau\le tn-t+\Delta.$$ Then
\begin{eqnarray}\label{eq10}
\hspace{2cm}e(D)&=& e(V_2\setminus N^+_{V_2}(u_0),V)+e(N^+_{V_2}(u_0),V)+e(V_1,V)\nonumber\\
&\le& (n-\Delta-\alpha_2)\Delta+t(n-1)+\alpha_2+tn-t+\Delta\nonumber\\
&\le& (n-\Delta-\alpha_2+2)\Delta+2t(n-1)\quad\quad\hspace{3.5cm}(since~\alpha_2\le \Delta)\nonumber\\
&=&-\left(\Delta-\frac{n+2-\alpha_2}{2}\right)^2+\frac{(n+2-\alpha_2)^2}{4}+2t(n-1)\nonumber\\
&\le& \frac{(n+2-\alpha_2)^2}{4}+2t(n-1).
\end{eqnarray}
It follows from \eqref{equ1} and (\ref{equ666}) that
\begin{equation}\label{eq1}
\frac{n^2-t^2}{4}+t(n-1)\le\frac{(n+2-\alpha_2)^2}{4}+2t(n-1).
\end{equation}
Solving the above inequality, we get
\begin{equation}\label{eq2}
\alpha_2\le n+2-\sqrt{n^2-t^2-4t(n-1)}\le 2t+3.
\end{equation}

		Now we show that
\begin{equation}\label{equ3}
\beta\ge \Delta-1.
\end{equation}
		Otherwise, if $\beta\le \Delta-2$, then
		\begin{eqnarray}\label{eq5}
			e(D)&= &\sum\limits_{u\in V_2\setminus\{v\}}d^+(u)+d^+(v)+e(V_1,V\setminus\{v\})+e(V_1,\{v\})\\
			&\le&  (n-\Delta-1)(\Delta-2)+\Delta+t(n-1)+\tau\nonumber\\
			&\le&  (n-\Delta-1)(\Delta-2)+\Delta+t(n-1)+\Delta\nonumber\\
			&=& -\left(\Delta-\frac{n+3}{2}\right)^2+\frac{n^2-2n+117}{4}+t(n-1)\nonumber\\
			&<&\phi(n,t)\nonumber,
		\end{eqnarray} a contradiction.

Combining (\ref{equ:3.3}), (\ref{eq2}) and (\ref{equ3}), we get an upper bound on the value of $\tau$:
\begin{equation}\label{eq7}
\tau\le 3t+4.
\end{equation}
We assert that
\begin{equation}\label{eq6}
\beta=\Delta. \tag{\ref{equ3}a}
\end{equation}
Suppose $\beta=\Delta-1$.   Using \eqref{eq7} and applying the same counting as in (\ref{eq5}), we have
\begin{eqnarray*}\label{}
			e(D)&= &\sum\limits_{u\in V_2\setminus\{v\}}d^+(u)+d^+(v)+e(V_1,V\setminus\{v\})+e(V_1,\{v\})\\
			&\le&  (n-\Delta-1)(\Delta-1)+\Delta+t(n-1)+\tau\nonumber\\
			&\le&  (n-\Delta-1)(\Delta-1)+\Delta+t(n-1)+3t+4\\
			&=& -\left(\Delta-\frac{n+1}{2}\right)^2+\frac{n^2-2n+21}{4}+tn+2t\nonumber\\
			&<&\phi(n,t)\nonumber,
		\end{eqnarray*} a contradiction.	Hence, we get (\ref{eq6}).

Next we improve  (\ref{eq8}). Substituting (\ref{eq7}) into (\ref{eq3}),  we obtain
\begin{equation}\label{eq11}
\frac{n-t}{2}-3<\Delta<\frac{n+t}{2}+3. \tag{\ref{eq8}a}
\end{equation}	
We claim that
\begin{equation}\label{eq9}
\alpha_2\le t.\tag{\ref{eq2}a}
\end{equation}
Recall  that $u_0\in V_2$ is a vertex  with $\alpha_2$ successors from $V_2$. By (\ref{equ2}), we give a new estimate of $e(N^+_{V_2}(u_0),V)$:
\begin{eqnarray*}
e(N^+_{V_2}(u_0),V)&\le& e(N^+_{V_2}(u_0),V_1)+e(N^+_{V_2}(u_0),V_2)\\
&\le&t\Delta+\alpha_2^2.
\end{eqnarray*}
Applying the same counting as in (\ref{eq10}), we have
\begin{eqnarray}
e(D)&=& e(V_2\setminus N^+_{V_2}(u_0),V)+e(N^+_{V_2}(u_0),V)+e(V_1,V)\nonumber\\
&\le& (n-\Delta-\alpha_2)\Delta+t\Delta+\alpha_2^2+tn-t+\tau\nonumber\\
&\le& -\Delta^2+(n-\alpha_2+t)\Delta+tn+\alpha_2^2-t+\tau\nonumber\\
&=&-\left(\Delta-\frac{n+t-\alpha_2}{2}\right)^2+\frac{(n+t-\alpha_2)^2}{4}+tn+\alpha_2^2-t+\tau\nonumber\\
&\le& \frac{(n+t-\alpha_2)^2}{4}+tn+\alpha_2^2-t+\tau\nonumber.
\end{eqnarray}
Suppose $\alpha_2\ge t+1$. Since  {$n\ge 17 t^2/2+30t+27$}, by \eqref{eq2} and \eqref{eq7} we have
$$n^2-(n+t-\alpha_2)^2\ge 2n-1> 4\alpha_2^2+t^2+4\tau,$$
which leads to
$$\frac{n^2-t^2}{4}+tn-tn> \frac{(n+t-\alpha_2)^2}{4}+tn+\alpha_2^2-t+\tau,$$
a contradiction. Hence, we get (\ref{eq9}).

By (\ref{eq9}) and \eqref{eq6}, it follows immediately from (\ref{equ:3.3}) that
\begin{equation}\label{eq14}
\tau\le 2t. \tag{\ref{eq7}a}
\end{equation}

Now we consider the precise value of $\Delta$. Since $D$ is $P_{t+1,2}$-free, we have
\begin{equation}\label{eq621}
e(V_2,V)\le \Delta(n-\Delta), \quad e(V_1)\le t\Delta,\quad e(V_1,V_2\setminus\{v\})\le t(n-\Delta-1).
\end{equation}
Let $a=\Delta(n-\Delta)-e(V_2,V)$, $b=t\Delta-e(V_1)$ and $c=t(n-\Delta-1)-e(V_1,V_2\setminus\{v\})$. It is clear that $a,b,c\ge 0$. By \eqref{equ1} and \eqref{equ666}, we have
\begin{eqnarray}\label{equ11}
a+b+c&\le& \Delta(n-\Delta)+t(n-1)+\tau-\phi(n,t)\\
&\le&\Delta(n-\Delta)+t(n-1)+\tau-\frac{n^2-t^2}{4}-t(n-1)\nonumber\\
&\le& \frac{t^2}{4}+2t.\nonumber
\end{eqnarray}

We assert that
\begin{equation}\label{equ7}
\alpha_1\le t.
\end{equation}
Otherwise, suppose $u\in V_1$ has $t+1$ successors from $V_2$, say $\{u_1,u_2,\ldots, u_{t+1}\}$. By \eqref{equ11}, we have
$$e(V_2,V)=\Delta(n-\Delta)-a\ge \Delta(n-\Delta)-\frac{t^2}{4}-2t,$$   which implies
 $$e(\{u_1,u_2,\ldots, u_{t+1}\},V)\ge (t+1)\Delta-\frac{t^2}{4}-2t.$$
 By \eqref{eq11} and \eqref{eq9}, we get $$e(\{u_1,u_2,\ldots, u_{t+1}\},V_1)\ge (t+1)\Delta-\frac{t^2}{4}-2t-t(t+1)>t\Delta.$$ It follows  that $u_1,u_2,\ldots, u_{t+1}$ share a common successor from $V_1$, say $w$. Then $D$ contains $t+1$ paths of length 2  from  $u$ to $w$, a contradiction.

Let $V':=\{w\in V_1|d^+_{V_1}(w)\ge \Delta-2t-t^2/4\}$. Since $e(V_1)\le t\Delta$ and {$n> t^3/2+9t^2/2+5t+7$}, by (\ref{eq11})  we have $|V'|\le t$.

 We Claim that
\begin{equation}\label{equ4}
e(V,V_1\setminus V')\ge (n-\Delta+t)(\Delta-|V'|)-\frac{t^3}{4}-\frac{9}{4}t^2-2t.
\end{equation}
Let $U=\{u_1,u_2,\ldots,u_i\}$ be the set of    vertices in $V_2$  with out-degrees less than $\Delta$. It is clear that $i\le a\le t^2/4+2t$. By \eqref{eq9} and \eqref{eq621}, we have $e(U,V_2)\le i\alpha_2\le it\le at$ and
\begin{eqnarray*}
e(U,V)&=&e(V_2,V)-e(V_2\setminus U, V)=\Delta(n-\Delta)-a-e(V_2\setminus U, V)\\
&\ge& \Delta(n-\Delta)-a-\Delta(n-\Delta-i)=i\Delta-a,
\end{eqnarray*}
which lead to
\begin{equation}\label{equ8}
e(U,V_1\setminus V')= e(U,V)-e(U,V_2)-e(U,V')\ge    i\Delta-a- at-i|V'|.
\end{equation}
For each vertex  $u\in V_2$ with $d^+(u)=\Delta$, we have
\begin{equation}\label{eq622}
d^+(w)\ge \Delta-2t-t^2/4 \quad{\rm for~~all}\quad w\notin N^+(u).
\end{equation} Otherwise, suppose $u_1\not\in N^+(u)$ with $d^+(u_1)<\Delta-2t-t^2/4$. Using the same arguments as for $\tau(v)\le 2t$, we get $\tau(u)\le 2t$.  Then we count the size of $D$ as
\begin{eqnarray}\label{eq623}
e(D)&=&e(N^+(u),V)+e(V\setminus N^+(u),V)\\
&=&e(N^+(u),V\setminus \{u\})+\tau(u)+\sum\limits_{w\notin N^+(u)\cup \{u_1\}}d^+(w)+d^+(u_1)\nonumber\\
&<&t(n-1)+2t+\Delta(n-\Delta)-2t-\frac{t^2}{4}\nonumber\\
&\le &\phi(n,t),\nonumber
\end{eqnarray} which contradicts $e(D)\ge \phi(n,t)$. Hence, we have   \eqref{eq622}, which leads to
    $V_1\setminus V'\subseteq N^+(u)$ for all $u\in V_2\setminus U$ and
\begin{equation}\label{equ9}
e(V_2\setminus U,V_1\setminus V')=(n-\Delta -i)|V_1\setminus V'|.
\end{equation}
Since $e(V_1)=t\Delta -b$, we have
\begin{equation}\label{equ10}
e(V_1,V_1\setminus V')\ge t|V_1\setminus V'|-b= t(|V_1|-|V'|)-b.
\end{equation}
Combining (\ref{equ8}), (\ref{equ9}), (\ref{equ10}) and $a+b\le t^2/4+2t$, we obtain (\ref{equ4}) as follows:
\begin{eqnarray*}
e(V,V_1\setminus V')&=&e(U,V_1\setminus V')+e(V_2\setminus U,V_1\setminus V')+e(V_1,V_1\setminus V')\\
&\ge&i\Delta-at-a-i|V'|+(n-\Delta-i)(|V_1|-|V'|)+t(|V_1|-|V'|)-b\\
&=&(n-\Delta+t)(|V_1|-|V'|)-at-(a+b)\\
&\ge&(n-\Delta+t)(|V_1|-|V'|)-t(t^2/4+2t)-(t^2/4+2t)\\
&=&(n-\Delta+t)(\Delta-|V'|)-\frac{t^3}{4}-\frac{9}{4}t^2-2t.
\end{eqnarray*}

\par
Since  {$n> t^3/2+9t^2/2+7t+6$}, by \eqref{eq11} we have
 $$\Delta-|V'|>\frac{t^3}{4}+\frac{9}{4}t^2+2t.$$ By the pigeonhole principal, there is a vertex $w\in V_1$ with $d^-(w)=n-\Delta+t$. Since $\Delta^+\ge \Delta^-$, we have $\Delta\ge n-\Delta+t$, which implies
\begin{equation}\label{equ667}
\Delta\ge \left\lceil\frac{n+t}{2}\right\rceil.\tag{\ref{eq8}b}
\end{equation}
On the other hand, by  (\ref{eq14}), we have
\begin{eqnarray*}
e(D)&=&e(V_2,V)+e(V_1,V\setminus\{v\})+\tau(v)\\
&\le&\Delta(n-\Delta)+t(n-1)+\tau\\
&\le&\Delta(n-\Delta)+t(n+1)\\
&=&-\left(\Delta-\frac{n}{2}\right)^2+\frac{n^2}{4}+t(n+1).
\end{eqnarray*}
Combining this with  \eqref{equ666} and (\ref{equ667}), we obtain the exact value of $\Delta$:
\begin{equation}\label{equ14}
		\Delta= \left\lceil\frac{n+t}{2}\right\rceil.\tag{\ref{eq8}{c}}
\end{equation}
		
Recalling the definitions of $a,b,c$, by   (\ref{eq14}) and (\ref{equ14}), we can modify \eqref{equ11} to get
\begin{equation}\label{equ15}
a+b+c\le \tau-t\le t.\tag{\ref{equ11}a}
\end{equation}

Finally, we show the last result of this lemma
\begin{equation}\label{eq13}
\min\limits_{d^+(u)=\Delta}\tau(u)\le t+1.
\end{equation}

To the contrary, we assume
\begin{equation}\label{equ12}
\min\limits_{d^+(u)=\Delta}\tau(u)\ge t+2.
\end{equation}
Then for each vertex $u$ with $d^+(u)=\Delta$, we have $|N^-(u)\cap N^+(u)|\ge t+2$.
Define $$S:=\{u\in V_1|d^+(u)\ge \Delta-t\} \quad{\rm and}\quad  V_1^*=V_1\setminus S.$$
 Since $\alpha_1\le t$, we have $$d^+_{V_1}(u)= d^+(u)-d^+_{V_2}(u)\ge \Delta-2t \quad {\rm for~all} \quad u\in S,$$ which implies $S\subseteq V'$ and   $|S|\le t$.

 Define $$F:=\{u\in V_2|d^+(u)=\Delta\}.$$
We claim that
\begin{equation}\label{eq624}
F\subseteq N^-(u) \quad {\rm for ~all}\quad  u\in V_1^*.
\end{equation}
  Otherwise, suppose there exist $u_1\in V_1^*$ and  $u_2\in F$  such that   $u_2\notin N^-(u_1)$.   Then $d^+(u_1)<\Delta-t$. Similarly as in \eqref{eq623}, by (\ref{eq14}), we have
\begin{eqnarray*}
e(D)
&=&e(N^+(u_2),V\setminus \{u_2\})+\tau(u_2)+\sum\limits_{w\notin N^+(u_2)\cup \{u_1\}}d^+(w)+d^+(u_1)\\
&<&t(n-1)+\tau(u_2)+\Delta(n-\Delta)-t\\
&\le&tn+\Delta(n-\Delta)\\
&\le &\phi(n,t),
\end{eqnarray*}which contradicts $e(D)\ge \phi(n,t)$.

By the definition of $a$ and ({\ref{equ11}a}), we obtain
\begin{equation}\label{equ13}
|V_2\setminus F|=|V_2|-|F|\le a\le t,
\end{equation}
and hence $|F|\ge  n-\Delta-t\ge 2.$

We assert that $v\in N^+(S)$. Otherwise, suppose $v\not\in N^+(S)$. Since $\tau(v)\ge t+2$,   $v$ has at least $t+2$ successors from $V_1^*$.  By \eqref{eq624}, we can find a $P_{t+1,2}$ in $D$, a contradiction.

Moreover, by (\ref{equ12}), each vertex $u$ in $F\setminus \{v\}$ has at least 2 predecessors from $V_2$ as $e(V_1,u)\le t$,
 i.e.,
\begin{equation}\label{equ92}
d^-_{V_2}(u)\ge 2\quad {\rm for~ all}\quad u\in F\setminus \{v\}.
\end{equation}

Let $$V_2^*=\{u\in F\setminus \{v\}|d^-_{V_1^*}(u)=t\}.$$
By the definition of $c$, there are at most $c$ vertices in $F$ satisfying $d^-_{V_1}(u)<t$. Note that $v$ is the only vertex in $F$ with more than $t$ predecessors in $V_1$. By (\ref{equ15}), we have
$$|\{u\in F\mid d^-_{V_1}(u)=t\}|\ge |F|-c-1\ge |V_2|-a-c-1\ge |V_2|-t-1.$$
Since $\alpha_1\le t$ and $|S|\le t$, we have $N^+_{V_2}(S)\le t^2$.
Recall that $d^-_{V_1}(u)\le t$ for all $u\in V\setminus \{v\}$. We get
$$|V_2^*|\ge |\{u\in F|d^-_{V_1}(u)=t\}|-|N^+_{V_2}(S)|\ge |V_2|-t-1-t^2\ge \left\lfloor\frac{n-t}{2}\right\rfloor-t^2-t-1.$$

Let $$V_3:=N^-_{V_2}(V_2^*)\quad{\rm and}\quad V_4:=N^-_{V_2}(V_3).$$
By \eqref{equ92}, we have $e(V_3,V_2^*)\ge 2|V_2^*|$, which implies $|V_3|\ge 2|V_2^*|/t$, as $\alpha_2\le t$.
It follows from $(\ref{equ13})$ that
 $$|F\cap V_3|\ge 2|V_2^*|/t-t.$$
  Similarly, we have $$e(V_4,V_3)\ge e(V_4,V_3\cap F)\ge 4|V_2^*|/t-2t,$$ and hence $$|V_4|\ge 4|V_2^*|/t^2-2. $$

  Since $\alpha_2\le t$, by  (\ref{equ13}), we have $e(V_2\setminus F, V_2)\le t^2$, which implies
   $$|\{u\in V_2^*\mid d^-_F(u)\ge 2\}|\ge |V^*_2|-t^2.$$
   Since \eqref{equ92} guarantees  that each vertex in $\{u\in V_2^*\mid d^-_F(u)\ge 2\}$  is the terminal vertex of  at least two paths of length 2 in $D[V_2]$, there are at least $2|V_2^*|-2t^2$ paths of length 2 from $V_4$ to $V_2^*$.
   On the other hand, since $\alpha_2\le t$ and $|V_2\setminus F|\le t$,   $D[V_2]$ contains at most $|V_2\setminus F|\alpha_2^2\le    t^3$ paths of length 2 with the initial vertices from $V_2\setminus F$. Hence, the number of 2-paths from $F$ to $V_2^*$ is at least
   $2|V_2^*|-2t^2-t^3>0$,  since  {$n\ge t^3+4t^2+3t+4$}.

   Therefore,  there exists a path $u_1u_2u_3$ in $D[V_2]$ with $u_1\in  F$ and $u_3\in V_2^*$. By the definition of $V_2^*$, $u_3$ has $t$ predecessors $w_1,w_2,\ldots,w_t$ from $ V_1^*.$   By \eqref{eq624},   $u_1$ is a common predecessor of  $w_1,w_2,\ldots,w_t$. Combining with the path $u_1u_2u_3$, we find  $t+1$ paths of length 2 from $u_1$ to $u_3$, a contradiction. Hence, we obtain (\ref{eq13}).
\end{proof}

Now we are ready to present the proof of Theorem \ref{th1}.

\bigskip	
{\noindent\bf Proof of Theorem \ref{th1}.}	Let $D$ be a $P_{t+1,2}$-free digraphs of order $n$ with $ex(n,P_{t+1,2})$ arcs. Without loss of generality, we assume that $\Delta(D)=\Delta^+(D)$. Let $v\in V$   with $d^+(v)=\Delta$ and $\tau(v)=\min\limits_{d^+(u)=\Delta}\tau(u)$. 	By Lemma \ref{le12}, we have $\tau(v)\le t+1$ and $\Delta=\lceil (n+t)/2\rceil$.  By Lemma \ref{le7}, we have $e(V_1,V\setminus \{v\})\le t (n-1)$.

Now we count the size of $D$ as
\begin{eqnarray*}
e(D)&=&e(V_1,V)+e(V_2,V)\\
&=&e(V_1,V\setminus \{v\})+\tau(v)+e(V_2,V) \\
&\le&t(n-1)+t+1+|V_2|\Delta\\
&=&t(n-1)+t+1+\left(n-\left\lceil\frac{n+t}{2}\right\rceil\right)\left\lceil\frac{n+t}{2}\right\rceil \\
&=& \left\lceil \frac{n+t}{2}\right\rceil\left\lfloor \frac{n-t}{2}\right\rfloor+tn+1.
\end{eqnarray*}Combining this with (\ref{equ666}), we obtain Theorem \ref{th1}. This completes the proof. \qed
	
{\bf Remark.}  When $ \lfloor({n-t})/{2} \rfloor$ is even, there is a gap between the upper bound and the lower bound on $ex(n,P_{t+1,2})$. We are inclined to conjecture that  $ex(n,P_{t+1,2})=g(n,t)-1$.


\section*{Acknowledgement}
This work was supported by the National Natural Science Foundation of China (No.12171323), Guangdong Basic and Applied Basic Research Foundation (No. 2022A1515011995).





\end{document}